\documentclass[letterpaper,11pt,oneside,reqno]{amsart}
\usepackage[utf8]{inputenc}
\usepackage[T1]{fontenc}
\usepackage[english]{babel}
\usepackage[square,numbers]{natbib}
\usepackage{amsmath,amssymb,amsfonts,amsthm}
\usepackage{mathtools} 
\usepackage{mathrsfs}  
\usepackage{bm}        
\usepackage{graphicx}
\usepackage{tikz-cd}
\usepackage[margin=1in]{geometry}
\usepackage{microtype}
\usepackage{enumitem}
\setlist{nosep}
\usepackage[hidelinks]{hyperref}
\theoremstyle{plain}
\newtheorem{theorem}{Theorem}[section]
\newtheorem{proposition}[theorem]{Proposition}

\newtheorem{corollary}[theorem]{Corollary}
\theoremstyle{definition}
\newtheorem{definition}[theorem]{Definition}

\newtheorem{remark}[theorem]{Remark}

\newcommand{\ZZ}{\mathbb{Z}}

\newcommand{\Gm}{\mathbb{G}_m}

\title[]{Bloch--Suslin Complex and Strong $\mathbb{A}^{1}-$invariance}
\author{Saheb Mohapatra}
\address{Department of Mathematics, Durham University, United Kingdom}
\email{saheb.mohapatra.math@gmail.com}
\subjclass[2020]{}
\begin{document}
\begin{abstract}
We prove that there is no extension of the abelian groups appearing in the Bloch-Suslin complex to strongly $\mathbb{A}^{1}-$invariant sheaves on $Sm_{k}$ (char($k$)=0) that also extend the canonical symbol maps from the respective $\mathbb{G}_{m}^{\wedge n}$ (i.e., from $\mathbb{G}_{m}$ and $\mathbb{G}_{m}^{\wedge2}$).
\end{abstract}
\maketitle







\section{Introduction}
Let \(k\) be a perfect field and write \(Sm_{k}\) for the category of smooth, separated schemes of finite type over \(k\). As in \cite{Morel12}, we consider Nisnevich sheaves on \(Sm_{k}\); denote by \(Shv(Sm_{k})\) (resp. \(Shv_{\mathcal{A}b}(Sm_{k})\)) the category of Nisnevich sheaves of sets (resp. abelian groups) on $Sm_{k}$. A central notion in this paper is strong $\mathbb{A}^{1}-$invariance for Nisnevich sheaves.  A sheaf of groups \(G\) is called \emph{strongly $\mathbb{A}^{1}-$invariant} if for every \(X\in Sm_{k}\) the projection $\mathbb{A}^{1}\times X\to X$ induces isomorphisms
\[
H^{i}_{\mathrm{Nis}}(X;G)\xrightarrow{\ \simeq\ } H^{i}_{\mathrm{Nis}}(\mathbb{A}^{1}\times X;G)
\qquad(i=0,1).
\]
A sheaf \(M\) of abelian groups is \emph{strictly $\mathbb{A}^{1}-$invariant} if the above morphism is an isomorphism for all \(i\ge 0\).  Denote by $SAI(Sm_{k})$ the subcategory of strongly $\mathbb{A}^{1}-$invariant sheaves in $Shv_{\mathcal{A}b}(Sm_{k})$. These notions originate in Morel--Voevodsky's \(\mathbb{A}^{1}\)-homotopy theory; and we refer to \cite{Morel12} primarily throughout the paper.\\

One of the canonical (and in fact, universal by Theorem \ref{2.3}) families of strongly, and hence strictly $\mathbb{A}^{1}-$invariant (see \cite{Morel12}) sheaves of abelian groups are the Milnor--Witt \(K\)-theory sheaves, \textbf{K}$_{n}^{MW}$. The Milnor $K-$theory sheaves, \textbf{K}$^{M}_{n}$, are certain quotients of these and are also strongly $\mathbb{A}^{1}-$invariant. The Milnor-Witt (resp. Milnor) $K-$groups were originally defined for fields via an explicit presentation as in Definition \ref{2.1}. As done in \cite{Morel12}, if $\mathcal{F}_{k}$ denotes the category of fields of finite transcendence degree over $k$, then by verifying certain properties of \textbf{K}$^{MW}_{n}$ as functors on $\mathcal{F}_{k}$, we can extend it to strongly $\mathbb{A}^{1}-$invariant sheaves on $Sm_{k}$. Denote by $\mathbb{G}_{m}^{\wedge n}$, the Nisnevich sheaf on $Sm_{k}$ sending an irreducible $X$ in $Sm_{k}$ to $(\mathcal{O}(X)^{*})^{\wedge n}$, pointed by $1$. Then, as shown in \cite{Morel12}, we have symbol maps $\sigma_{n}:\mathbb{G}_{m}^{\wedge n}\rightarrow$\textbf{K}$^{MW}_{n}$. \\

On the other hand, the abelian groups appearing in the Bloch--Suslin complex arise from scissors-congruence and dilogarithmic constructions and conjecturally appear in low-weight motivic complexes of Goncharov (see \cite{Goncharov95}). For a field \(F\) the pre-Bloch group (sometimes called the scissors congruence group) \(\mathcal{P}(F)\) is the quotient of the free abelian group $\ZZ[F^{**}]$ on \(F^{**}:=F^{*}\setminus\{1\}\) by the subgroup generated by the classical five-term relations
\[
\{x\}-\{y\}+\Big\{\frac{y}{x}\Big\}+\Big\{\frac{1-x}{1-y}\Big\}-\Big\{\frac{y(1-x)}{x(1-y)}\Big\},
\qquad x\ne y,
\]
and there is an exact sequence relating \(\mathcal{P}(F)\), \(\wedge^{2}_{\ZZ}F^{*}\), and \(K_{2}^{M}(F)\). Here, $\wedge^{2}_{\ZZ}F^{*}$ denotes the quotient of $F^{*}\otimes_{\mathbb{Z}} F^{*}$ by the subgroup generated by $x\otimes y+y\otimes x$ for $x,y\in F^{*}$. Denote by $\mathcal{A}$, the  functor on $\mathcal{F}_{k}$ sending $F$ to $\wedge_{\ZZ}^{2}F^{*}$. Formally, the Bloch-Suslin complex is the following complex of functors on $\mathcal{F}_{k}$
\[
\mathcal{P}\xrightarrow{\lambda}\mathcal{A}
\]
where for $F\in\mathcal{F}_{k}$, $\lambda(F):\mathcal{P}(F)\rightarrow\mathcal{A}(F)$ is given by $\{z\}\mapsto z\wedge(1-z)$. The cokernel of this map is isomorphic to $K^{M}_{2}(F)$ and the kernel is rationally isomorphic to $K_{3}^{\mathrm{ind}}(F)$. The latter was proved by Suslin (see \cite{Suslin91}). The former was proved by Hutchinson (see \cite{Hutchinson90}) by analyzing the spectral sequence associated to a double  complex computing $H_{2}(GL_{2}(F),\ZZ)$ which yields the following exact sequence:
\[
\mathcal{P}(F)\xrightarrow{\ \lambda\ } \wedge^{2}_{\ZZ}F^{*} \xrightarrow{\ \mu\ } K_{2}^{M}(F)\to 0,
\]
where \(\lambda(\{z\})=z\wedge(1-z)\) and \(\mu(x\wedge y)=\{x,y\}\).  The groups \(\mathcal{P}(F)\) and \(\wedge^{2}_{\ZZ}F^{*}\), however, are \emph{a priori} functors on the category \(\mathcal{F}_{k}\).  A natural question is whether, like Minor $K-$theory, these functors admit extensions to Nisnevich sheaves on \(Sm_{k}\) which are strongly $\mathbb{A}^{1}-$invariant while simultaneously extending the evident symbol maps from $\mathbb{G}_{m}$ (resp. $\mathbb{G}_{m}^{\wedge 2}$):
\[
\Gm\longrightarrow \mathcal{P}\quad (resp.\quad  \Gm^{\wedge 2}\longrightarrow \mathcal{A})\]
restricting on fields to \(z\mapsto\{z\}\) for $z\neq1$ (and $1\mapsto 0$) (resp. $(x,y)\mapsto x\wedge y$). Another motivation for this question is the above exact sequence combined with the fact that \textbf{K}$^{M}_{2}\in SAI(Sm_{k})$, an abelian subcategory of $Shv_{\mathcal{A}b}(Sm_{k})$ (see \cite{Morel12}).

\medskip

The main result of this paper answers this question in the negative.

\begin{theorem}\label{thm:main}
Let \(k\) be a perfect field of characteristic zero.  There do not exist Nisnevich sheaves of abelian groups $\mathcal{P}$ and $\mathcal{A}$ on \(Sm_{k}\) which are strongly (hence strictly) $\mathbb{A}^{1}-$invariant and whose evaluations on fields recover \(\mathcal{P}\) and \(\mathcal{A}\), together with maps
$
\mathbb{G}_{m}\to \mathcal{P},  \mathbb{G}_{m}^{\wedge 2}\to \mathcal{A},
$
that restrict on fields to the canonical symbol maps \(z\mapsto\{z\}\) for $z\neq1$ $(1\mapsto 0)$ and \((x,y)\mapsto x\wedge y\).
\end{theorem}

The main ingredient of the proof is Theorem \ref{2.3} which compares objects of $SAI(Sm_{k})$ with the universal object \textbf{K}$^{MW}_{n}$ with respect to the symbol map $\sigma_{n}$. Thus, although \(K_{2}^{M}\) extends to a strongly $\mathbb{A}^{1}-$invariant sheaf, the intermediate objects \(\mathcal{P}\) and \(\mathcal{A}\) can not be promoted to strongly $\mathbb{A}^{1}-$invariant sheaves if one insists on extending the canonical symbol maps from \(\Gm\) and \(\Gm^{\wedge2}\).  Intuitively, the obstruction arises because \(\mathcal{P}\) and \(\mathcal{A}\) encode subtle five-term/dilogarithmic relations and anti-symmetric tensor structure that are incompatible with the relations in Milnor-Witt $K-$theory.

\subsection*{Acknowledgements}
This work was carried out as part of the author’s master’s thesis at the Indian Statistical Institute, Kolkata, India, from January to April 2025. The author is deeply grateful to his supervisor, Dr. Utsav Choudhury, for suggesting the problem and for his constant support throughout the project.

\bigskip

\section{Preparatory steps}

\begin{definition}[Milnor-Witt K-groups]\label{2.1}\cite{Morel12}
Consider the graded associative ring generated by the symbols $[u]$ for each $u\in F^{*}$ and one symbol $\eta$ of degree (-1) with the following relations:
\begin{itemize}
\item (Steinberg Relation) $\forall a\in F^{*}-\{1\},[a][1-a] = 0$
\item $\forall (a,b)\in (F^{*})^{2}, [ab] = [a]+[b]+\eta[a][b]$
\item $[u]\eta=\eta[u], \forall u\in F^{*}$
\item $h:=\eta[-1]+2$, then $\eta h = 0$.
\end{itemize}
Denote this ring by $K_{*}^{MW}(F)$ and its $n^{th}$ degree part by the abelian group $K_{n}^{MW}(F)$. So, we have $[a(1-a)]=[a]+[1-a]$ in $K_{1}^{MW}(F)$. Milnor $K-$theory, $K_{*}^{M}$, is the quotient of $K_{*}^{MW}(F)$ by the two-sided ideal generated by $\eta$ and $K_{n}^{M}(F)$ is its $n^{th}$ degree part. 
\end{definition}
\begin{definition}
Fix an irreducible $X\in Sm_{k}$ with function field $F$. As $X$ is irreducible, $(O(X)^{*})^{\wedge n}\subset(F^{*})^{\wedge n}$, where for any pointed set $(A,a)$, $A^{\wedge n}$ denotes its $n-$fold smash product with itself pointed by $a$. Here, $\mathcal{O}(X)^{*}$ is pointed by $1$. So, we have a map $(O(X)^{*})^{\wedge n}\rightarrow K_{n}^{MW}(F)$ such that $(u_{1},...,u_{n})\mapsto[u_{1}]...[u_{n}]$. But $[u_{1}]...[u_{n}]\in$\textbf{K}$_{n}^{MW}(X)$ by its definition in \cite{Morel12}.
This map, $\sigma_{n}:\mathbb{G}_{m}^{\wedge n}\rightarrow$\textbf{K}$_{n}^{MW}$ (called the canonical symbol map) is a morphism of pointed sheaves on $Sm_{k}$.
\end{definition}
\begin{theorem}\label{2.3}\cite{Morel12}
Let $n\geq 1$. The morphism $\sigma_{n}$ is the universal morphism from $\mathbb{G}_{m}^{\wedge n}$ to a strongly $\mathbb{A}^{1}$-invariant sheaf of abelian groups. That is, given a morphism of pointed sheaves $\phi:\mathbb{G}_{m}^{\wedge n}\rightarrow M$ where $M$ is a strongly $\mathbb{A}^{1}$-invariant sheaf of abelian groups, there exists a unique morphism of sheaves of abelian groups $\Phi$ such the following diagram commutes:

\[\begin{tikzcd}
{\mathbb{G}_{m}^{\wedge n}} \\
\\
{\textbf{K}_{n}^{MW}} &&& M
\arrow["\sigma_{n}", from=1-1, to=3-1]
\arrow["\phi", from=1-1, to=3-4]
\arrow["\Phi"', from=3-1, to=3-4]
\end{tikzcd}\]
\end{theorem}
\begin{theorem}\label{2.4}\cite{Morel12}
Suppose $M$ is a strongly $A^{1}$-invariant sheaf of abelian groups on $Sm_{k}$. Let $n\geq 1$ be an integer, and let $\phi:\mathbb{G}_{m}^{\wedge n}\rightarrow M$ be a morphism of pointed sheaves. Then, for any $F\in \mathcal{F}_{k}$, there is a unique morphism $\Phi_{n}(F):K_{n}^{MW}(F)\rightarrow M(F)$ such that for any $(u_{1},...,u_{n})\in (F^{*})^{n}$, $\Phi_{n}(F)([u_{1}]...[u_{n}])=\phi(u_{1},...,u_{n})$.
\end{theorem}
For two categories $\mathcal{C}$ and $\mathcal{D}$, denote by $Fct(\mathcal{C},\mathcal{D})$, the category of functors from $\mathcal{C}$ to $\mathcal{D}$.
\begin{definition}[Scissor congruence group]\cite{Suslin91}
    Let $F$ be a field. Denote by $F^{**}:=F^{*}-\{1\}$. Define $\mathcal{P}(F)$ to be the quotient of the free abelian group $\mathbb{Z}[F^{**}]$ on the set $F^{**}$ by the subgroup generated by the relations $\{x\}-\{y\}+\{\frac{y}{x}\}+\{\frac{1-x}{1-y}\}-\{\frac{y(1-x)}{x(1-y)}\}; x\neq y$. So, $\mathcal{P}\in Fct(\mathcal{F}_{k},\mathcal{A}b)$. For convenience of notation, we denote the $0$ in $\mathcal{P}(F)$ by $\{1\}$.
\end{definition}
\begin{remark}\label{2.6}
    \begin{itemize}
    \item We have $\phi\in Hom_{Fct(\mathcal{F}_{k},\mathcal{A}b)}(\mathbb{G}_{m},\mathcal{P})$ defined on $\mathbb{G}_{m}(F)=F^{*}\rightarrow\mathcal{P}(F)$ by $z\mapsto \{z\}$. So, $\phi(F)(1)=\{1\}=0$, as per our notation.
        \item Define $\mathcal{A}\in Fct(\mathcal{F}_{k},\mathcal{A}b)$ defined by $\mathcal{A}(F):=\wedge^{2}_{\mathbb{Z}}F^{*}$. Similarly, we have $\psi\in Hom_{Fct(\mathcal{F}_{k},\mathcal{A}b)}(\mathbb{G}^{\wedge 2}_{m},\mathcal{A})$ defined on $\mathbb{G}_{m}^{\wedge 2}(F)=(F^{*})^{\wedge 2}\rightarrow \mathcal{A}(F)=\wedge^{2}_{\mathbb{Z}}(F^{*})$ by $(a,b)\mapsto a\wedge b$. This makes sense as $a\wedge1=1\wedge b=0$ in $\wedge^{2}_{\mathbb{Z}}(F^{*}), \forall (a,b)\in F^{*}\times F^{*}$.
    \end{itemize}
\end{remark}
The following proposition has been taken from \cite{Suslin91}.
\begin{proposition}\label{2.7}
    For any field $F$, the element $\{x\}+\{1-x\}$ of $\mathcal{P}(F)$ is independent of the choice of $x\in F^{**}$. Call this element $c_{F}$. Moreover, if $\sqrt{-1},\sqrt{-3},2^{-1}\in F$, then $c_{F}=0$ in $\mathcal{P}(F)$.
\end{proposition}
\begin{proof}
From the five term relation in $\mathcal{P}(F)$ (and by replacing the pair $(x,y)$ by $(1-y,1-x)$), we get
\begin{enumerate}
    \item $\{x\}-\{y\}+\{\frac{y}{x}\}+\{\frac{1-x}{1-y}\}-\{\frac{y(1-x)}{x(1-y)}\}=0$
    \item $\{1-y\}-\{1-x\}+\{\frac{1-x}{1-y}\}+\{\frac{y}{x}\}-\{\frac{y(1-x)}{x(1-y)}\}=0$
\end{enumerate}

Subtracting $(2)$ from $(1)$, we get $\{x\}+\{1-x\}=\{y\}+\{1-y\}$ for all $x,y\in F^{**}$. We call this element $c_{F}$.\\

Similarly, consider the element $\{x\}+\{x^{-1}\}$ in $\mathcal{P}(F)$. Replacing $(x,y)$ by $(x^{-1},y^{-1})$ in the five term relation, we get
\begin{itemize}
    \item $\{x\}-\{y\}+\{\frac{y}{x}\}+\{\frac{1-x}{1-y}\}-\{\frac{y(1-x)}{x(1-y)}\}=0$
    \item $\{x^{-1}\}-\{y^{-1}\}+\{\frac{x}{y}\}+\{\frac{y(1-x)}{x(1-y)}\}-\{\frac{1-x}{1-y}\}=0$
\end{itemize}
Adding these, we get $<x>-<y>+<\frac{y}{x}>=0$ where $<z>:=\{z\}+\{z^{-1}\}$ for any $z\in F^{**}$. Replacing $x$ by $y$, we have $<y>-<x>+<\frac{x}{y}>=0$. Since $<\frac{x}{y}>=<\frac{y}{x}>$, we get $2<\frac{x}{y}>=0$. Hence, $\forall z\in F^{**}$, $2<z>=0$ and $\forall x,y\in F^{**}$, $<x>+<y>=<xy>$.\\

$3c_{F}=\{x\}+\{1-x\}+\{x^{-1}\}+\{1-x^{-1}\}+\{(1-x)^{-1}\}+\{1-(1-x)^{-1}\}$\\$=<x>+<1-x>+<1-x^{-1}>=<-(1-x)^{2}>=<-1>$. So, if $z:=\sqrt{-1}\in F$, $3c_{F}=<z^{2}>=0$. 
If $\sqrt{-3},2^{-1}\in F$, $\exists z\in F: 1-z=z^{-1}$. So, $c_{F}=\{z\}+\{1-z\}=<z>\implies 2c_{F}=0$. Combining both observations, if $\sqrt{-1},\sqrt{-3},2^{-1}\in F$, $c_{F}=0$.

\end{proof}

\begin{proposition}\label{2.8}
Let $F$ be a field and $x,y\in F^{*}$ such that $x^{a}=y^{b}$ only when $(a,b)=(0,0)\in\mathbb{Z}^{2}$, then $x\wedge y\neq0$ in $\wedge^{2}_{\mathbb{Z}}(F^{*})$.    
\end{proposition}
\begin{proof}
    Let $G:=<x,y>$, the subgroup of $F^{*}$ generated by $x$ and $y$. So, $G$ is a free abelian group of rank 2. Now, consider the following commutative diagram induced by the inclusion $i:G\rightarrow F^{*}$:

\[\begin{tikzcd}
	{\wedge^{2}_{\mathbb{Z}}G} && {\wedge^{2}_{\mathbb{Z}}F^{*}} \\
	{\wedge^{2}_{\mathbb{Z}}G\bigotimes_{\mathbb{Z}}\mathbb{Q}} && {\wedge^{2}_{\mathbb{Z}}F^{*}\bigotimes_{\mathbb{Z}}\mathbb{Q}} \\
	{\wedge^{2}_{\mathbb{Q}}(G\otimes_{\mathbb{Z}}\mathbb{Q})} && {\wedge^{2}_{\mathbb{Q}}(F^{*}\otimes_{\mathbb{Z}}\mathbb{Q})}
	\arrow["{\wedge^{2}_{\mathbb{Z}}(i)}"', from=1-1, to=1-3]
	\arrow[from=1-1, to=2-1]
	\arrow[from=1-3, to=2-3]
	\arrow[from=2-1, to=2-3]
	\arrow["\cong"', from=2-1, to=3-1]
	\arrow["\cong", from=2-3, to=3-3]
	\arrow[hook, from=3-1, to=3-3]
\end{tikzcd}\]

Now, the bottom arrow is injective as $G\otimes_{\mathbb{Z}}\mathbb{Q}$ is $\mathbb{Q}$-subspace of the vector space $F^{*}\otimes_{\mathbb{Z}}\mathbb{Q}$. Also, note that the one dimensional $\mathbb{Q}$-vector space $\wedge^{2}_{\mathbb{Q}}(G\otimes_{\mathbb{Z}}\mathbb{Q})$ is spanned by $(x\otimes_{\mathbb{Z}}1)\wedge_{\mathbb{Q}} (y\otimes_{\mathbb{Z}}1)$. If $x\wedge y=0$ in $\wedge^{2}_{\mathbb{Z}}F^{*}$, $(x\otimes_{\mathbb{Z}}1)\wedge_{\mathbb{Q}} (y\otimes_{\mathbb{Z}}1)=0$ in $\wedge^{2}_{\mathbb{Q}}(F^{*}\otimes_{\mathbb{Z}}\mathbb{Q})$, a contradiction to the bottom map being injective.

\end{proof}
\section{Main theorems}
\begin{theorem}\label{3.1}
For any field $k$ with char($k$)=0, the functor $\mathcal{P}$ on $\mathcal{F}_{k}$ can not be extended to an element of SAI($Sm_{k}$) such that the map $\phi:\mathbb{G}_{m}\rightarrow \mathcal{P}$ in $Hom_{Fct(\mathcal{F}_{k},\mathcal{A}b)}(\mathbb{G}_{m},\mathcal{P})$ as in Remark \ref{2.6} extends to a map between the pointed sheaves on $Sm_{k}$: $\mathbb{G}_{m}\rightarrow\mathcal{P}$ where the latter is pointed by $0$.
\end{theorem}
\begin{proof}
    Let us assume to the contrary that $\mathcal{P}$ can be extended to such an element of $SAI(Sm_{k})$. Then, by Theorem \ref{2.3}, $\exists \Phi\in Hom_{Shv_{\mathcal{A}b}(Sm_{k})}(K_{1}^{MW},\mathcal{P})$ such that the following is a commutative diagram in $Shv(Sm_{k})$:

\[\begin{tikzcd}
	{\mathbb{G}_{m}} \\
	{K^{MW}_{1}} && {\mathcal{P}}
	\arrow["{\sigma_{1}}"', from=1-1, to=2-1]
	\arrow["\phi", from=1-1, to=2-3]
	\arrow["\Phi"', from=2-1, to=2-3]
\end{tikzcd}\]

Now, let $K:=k(\sqrt{-1},\sqrt{-3},\sqrt{-11})$, an object of $\mathcal{F}_{k}$. By Proposition \ref{2.7}, $c_{K}=0$ in $\mathcal{P}(K)$. By Theorem \ref{2.4}, we have $\Phi(K):K^{MW}_{1}(K)\rightarrow \mathcal{P}(K)$ and $\Phi(K)([z])=\{z\}$. Now, by Remark \ref{2.6}, for any $z\in K^{**}$, $[z(1-z)]=[z]+[1-z]$ in $K^{MW}_{1}(K)$. Under $\Phi$ this maps to $\{z(1-z)\}=c_{K}=0$ in $\mathcal{P}(K)$. Since $\sqrt{-11}\in K$, $\exists w\in K$, a solution of the quadratic equation $z(1-z)=z-z^{2}=3$. So, $\{3\}=0$ in $\mathcal{P}(K)$. $\{3\}$ maps to $3\wedge(-2)$ under the map $\lambda:\mathcal{P}(K)\rightarrow \wedge^{2}_{\mathbb{Z}}{F}$. As $\mathrm{char}(K)=0$, $3^{a}=(-2)^{b}$ for $(a,b)\in\mathbb{Z}\times\mathbb{Z}$ only when $a=0=b$. By Proposition \ref{2.8}, $3\wedge(-2)\neq 0$ in $\wedge^{2}_{\mathbb{Z}}(K^{*})$, a contradiction.
    \end{proof}
\begin{theorem}\label{3.2}
    For any field $k$ with char($k$)=0, the functor $\mathcal{A}$ on $\mathcal{F}_{k}$ can not be extended to an element of SAI($Sm_{k}$) such that the map $\psi:\mathbb{G}^{\wedge2}_{m}\rightarrow \mathcal{A}$ in $Hom_{Fct(\mathcal{F}_{k},\mathcal{A}b)}(\mathbb{G}_{m},\mathcal{A})$ as in Remark \ref{2.6} extends to a map between the pointed sheaves on $Sm_{k}$: $\mathbb{G}_{m}^{\wedge 2}\rightarrow\mathcal{A}$ where the latter is pointed by $0$.
\end{theorem}
\begin{proof}
    Let us assume to the contrary that $\mathcal{A}$ can be extended to such an element of $SAI(Sm_{k})$. By Theorem \ref{2.3}, $\exists \Psi\in Hom_{Shv_{\mathcal{A}b}(Sm_{k})}(K^{MW}_{2},\mathcal{P})$ such that the following is a commutative diagram in $Shv(Sm_{k})$:

\[\begin{tikzcd}
	{\mathbb{G}^{\wedge2}_{m}} \\
	{K^{MW}_{2}} && {\mathcal{A}}
	\arrow["{\sigma_{2}}"', from=1-1, to=2-1]
	\arrow["\psi", from=1-1, to=2-3]
	\arrow["\Psi"', from=2-1, to=2-3]
\end{tikzcd}\]
By Definition \ref{2.1}, $[3][-2]=0$ in $K^{MW}_{2}(k)$ but maps to $3\wedge(-2)$ under $\Psi(k)$ as in Theorem \ref{2.4}. As in the proof of the previous theorem, since $\mathrm{char}(k)=0$, $3\wedge(-2)\neq0$ in $\mathcal{A}(k)=\wedge^{2}_{\mathbb{Z}}(k^{*})$, a contradiction.
\end{proof}
\begin{corollary}
    There is no $S\in SAI(Sm_{k})$ such that $S(X)=\mathcal{P}(\mathcal{O}(X)^{*})$ or $S(X)=\mathcal{P}(F(X))$ where $X\in Sm_{k}$ is irreducible with function field $F(X)$ (and analogously for $\mathcal{A}$). Here, $\mathcal{P}(\mathcal{O}(X)^{*})$ is defined to be the quotient of the free abelian group $\mathbb{Z}[\mathcal{O}(X)^{*}-\{1\}]$ by the subgroup generated by the five-term relations whenever they make sense in $\mathcal{O}(X)^{*}$.
\end{corollary}
\begin{proof}
    Let $X$ be an irreducible scheme in $Sm_{k}$. To the contrary, if there is such an $S$, since $F(X)^{*}=\varinjlim_{U\subset X}\mathcal{O}(U)^{*}$ (running over open subschemes of $X$), we have maps between the pointed sheaves $\mathbb{G}_{m}$ and $S$ (pointed by $0$) which restrict to the canonical symbol maps. This holds in both cases, $S(X)=\mathcal{P}(\mathcal{O}(X)^{*})$ and $S(X)=\mathcal{P}(F(X))$ and contradicts Theorem \ref{3.1}. The analogous statement for $\mathcal{A}$ follows similarly from Theorem \ref{3.2}.
\end{proof}

\end{document}